\documentclass{amsart} 
\usepackage{amscd,amssymb,latexsym,subfigure,verbatim,hyperref,epsfig,amsmath,supertabular}
\usepackage[graph,frame,poly,arc]{xy}
\usepackage{calc}
\usepackage{graphicx}

\begin{document}

\newcommand{\mmbox}[1]{\mbox{${#1}$}}
\newcommand{\proj}[1]{\mmbox{{\mathbb P}^{#1}}}
\newcommand{\affine}[1]{\mmbox{{\mathbb A}^{#1}}}
\newcommand{\Ann}[1]{\mmbox{{\rm Ann}({#1})}}
\newcommand{\caps}[3]{\mmbox{{#1}_{#2} \cap \ldots \cap {#1}_{#3}}}
\newcommand{\N}{{\mathbb N}}
\newcommand{\Z}{{\mathbb Z}}
\newcommand{\Q}{{\mathbb Q}}
\newcommand{\R}{{\mathbb R}}
\newcommand{\K}{{\mathbb K}}
\newcommand{\p}{{\mathbb P}}
\newcommand{\A}{{\mathcal A}}

\newcommand{\CC}{{\mathcal C}}
\newcommand{\C}{{\mathbb C}}
\newcommand{\CR}{C^r(\hat P)}
\newcommand{\Co}{C^0(\hat P)}
\newcommand{\vs}{\vspace{5pt}}
\newcommand{\An}{\mathop{\rm A_n}\nolimits}
\newcommand{\Aroot}{\mathop{\rm A}\nolimits}
\newcommand{\Bn}{\mathop{\rm B}\nolimits}
\newcommand{\Dn}{\mathop{\rm D}\nolimits}
\newcommand{\Tor}{\mathop{\rm Tor}\nolimits}
\newcommand{\ann}{\mathop{\rm ann}\nolimits}
\newcommand{\Ext}{\mathop{\rm Ext}\nolimits}
\newcommand{\depth}{\mathop{\rm depth}\nolimits}
\newcommand{\E}{\mathop{\rm E}\nolimits}
\newcommand{\Hom}{\mathop{\rm Hom}\nolimits}
\newcommand{\im}{\mathop{\rm im}\nolimits}
\newcommand{\codim}{\mathop{\rm codim}\nolimits}
\newcommand{\pdim}{\mathop{\rm pdim}\nolimits}
\newcommand{\rank}{\mathop{\rm rank}\nolimits}
\newcommand{\conv}{\mathop{\rm conv}\nolimits}
\newcommand{\cone}{\mathop{\rm cone}\nolimits}
\newcommand{\st}{\mathop{\rm star}\nolimits}
\newcommand{\supp}{\mathop{\rm supp}\nolimits}
\newcommand{\arrow}[1]{\stackrel{#1}{\longrightarrow}}
\newcommand{\CB}{Cayley-Bacharach}
\newcommand{\coker}{\mathop{\rm coker}\nolimits}
\sloppy
\newtheorem{defn0}{Definition}[section]
\newtheorem{prop0}[defn0]{Proposition}
\newtheorem{conj0}[defn0]{Conjecture}
\newtheorem{que0}[defn0]{Question}
\newtheorem{thm0}[defn0]{Theorem}
\newtheorem{lem0}[defn0]{Lemma}
\newtheorem{corollary0}[defn0]{Corollary}
\newtheorem{example0}[defn0]{Example}
\newtheorem{rmk0}[defn0]{Remark}

\newenvironment{defn}{\begin{defn0}}{\end{defn0}}
\newenvironment{prop}{\begin{prop0}}{\end{prop0}}
\newenvironment{que}{\begin{que0}}{\end{que0}}
\newenvironment{conj}{\begin{conj0}}{\end{conj0}}
\newenvironment{thm}{\begin{thm0}}{\end{thm0}}
\newenvironment{lem}{\begin{lem0}}{\end{lem0}}
\newenvironment{cor}{\begin{corollary0}}{\end{corollary0}}
\newenvironment{exm}{\begin{example0}\rm}{\end{example0}}
\newenvironment{rmk}{\begin{rmk0}\rm}{\end{rmk0}}

\newcommand{\defref}[1]{Definition~\ref{#1}}
\newcommand{\propref}[1]{Proposition~\ref{#1}}
\newcommand{\queref}[1]{Question~\ref{#1}}
\newcommand{\thmref}[1]{Theorem~\ref{#1}}
\newcommand{\lemref}[1]{Lemma~\ref{#1}}
\newcommand{\corref}[1]{Corollary~\ref{#1}}
\newcommand{\exref}[1]{Example~\ref{#1}}
\newcommand{\secref}[1]{Section~\ref{#1}}
\newcommand{\poina}{\pi({\mathcal A}, t)}
\newcommand{\poinc}{\pi({\mathcal C}, t)}
\newcommand{\std}{Gr\"{o}bner}
\newcommand{\jq}{J_{Q}}

\title[Equivariant Chow cohomology of nonsimplicial toric varieties]%
 {Equivariant Chow cohomology \\of nonsimplicial toric varieties}

\author{Hal Schenck}
\thanks{Schenck supported by  NSF 0707667}\address{Schenck: Mathematics Department \\ University of Illinois Urbana-Champaign\\
  Urbana \\ IL 61801\\ USA}
\email{schenck@math.uiuc.edu}


\begin{abstract}
\noindent For a toric variety $X_\Sigma$ determined by a polyhedral fan $\Sigma \subseteq N$, 
Payne shows that the equivariant Chow cohomology is the $Sym(N)$--algebra $C^0(\Sigma)$ of
integral piecewise polynomial functions on $\Sigma$. We use the Cartan-Eilenberg spectral
sequence to analyze the associated reflexive sheaf $\mathcal{C}^0(\Sigma)$ on $\p_{\Q}(N)$,
showing that the Chern classes depend on subtle geometry of $\Sigma$ and giving 
criteria for the splitting of $\mathcal{C}^0(\Sigma)$ as a sum of line bundles. 
For certain fans associated to the reflection arrangement $\An$, we describe 
a connection between $C^0(\Sigma)$ and logarithmic vector fields tangent to $\An$. 
\end{abstract}
\maketitle
\vskip -1in
\section{Introduction}\label{sec:intro}
In \cite{bdp}, Bifet-De-Concini-Procesi show that the
integral equivariant cohomology ring $H_T^*(X_\Sigma)$ of a smooth toric 
variety $X_\Sigma$ is isomorphic to the integral Stanley-Reisner 
ring $A_\Sigma$ of the unimodular fan $\Sigma$, and in 
\cite{b2}, Brion shows that for $\Sigma$ simplicial, 
the rational equivariant Chow ring $A_T^*(X_\Sigma)_{\Q}$ is isomorphic 
to the ring of rational piecewise polynomial functions 
$C^0(\Sigma)_{\Q}$. A result of Billera \cite{b} shows that 
for a simplicial fan, $C^0(\Sigma)_{\Q}$ is isomorphic to the
rational Stanley-Reisner ring of the fan, so Brion's 
result is similar in spirit to \cite{bdp}. Brion and Vergne completed 
the picture for the simplicial case by showing in \cite{bv} that 
\[
A_T^*(X_\Sigma)_{\Q} \simeq H_T^*(X_\Sigma)_{\Q}.
\]
Integral cohomology is more delicate, in \cite{p},  Payne exhibits a 
complete toric surface with 2-torsion in $H_T^3$.
For a nonsimplicial fan $\Sigma$ there is no Stanley-Reisner ring,
but the results of Billera and Brion suggest that $C^0(\Sigma)$ 
could serve as a  possible substitute. In \cite{p} Payne proves that 
the integral equivariant Chow ring $A^*_T(X_\Sigma)$ does indeed satisfy
\[
A^*_T(X_\Sigma) \simeq  C^0(\Sigma).
\]
We analyze the $S=Sym_{\Q}(N)$-module structure of $C^0(\Sigma)_{\Q}$ and the
associated reflexive sheaf $\mathcal{C}^0(\Sigma)$ on $\p_{\Q}(N)$. 
In contrast to the simplicial case, where $C^0(\Sigma)_{\Q}$ 
is completely determined by the combinatorics of $\Sigma$, in the nonsimplicial
case there are surprising and subtle contributions from the geometry of $\Sigma$.
We begin by defining a cellular chain complex $\mathcal{C}$ whose 
top homology module is $C^0(\Sigma)$, then study the lower homology modules of the complex.
We give a complete description of the loci of the top dimensional support of the modules
$H_i(\mathcal{C})$. Using a Cartan-Eilenberg spectral sequence, we also obtain 
sufficient criteria for the splitting of $\mathcal{C}^0(\Sigma)$ as 
a sum of line bundles. \pagebreak

\subsection{Main results} \mbox{ }
\vskip .04in
\noindent{\bf Theorem 1} For all $i \ge 1$, $H_{d-i}(\mathcal{C})$ is supported in codimension at least $i+1$, and all associated primes of codimension $i+1$ are linear. 
\vskip .04in
\noindent In Theorem~\ref{minassPrimes} we give a combinatorial description of the 
codimension $i+1$ associated primes of $H_{d-i}(\mathcal{C})$ in terms of 
$\Sigma_{d-i}$ and $\Sigma_{d-i+1}$. 
\vskip .04in
\noindent{\bf Theorem 2} If $H_i(\mathcal{C})$ is either 
Cohen-Macaulay of codimension $d-i+1$ or zero for all $i<d$, 
then $C^0(\Sigma)$ is free.
\vskip .04in
\noindent In particular, $C^0(\Sigma)$ can be free even when some 
lower homology modules are nonzero, whereas
in the simplicial case, $C^0(\Sigma)$ is free 
iff $H_{d-i}(\mathcal{C})=0$ for all $i\ge 1$.
\begin{exm}\label{exm:first}\cite{ms}
Let $\Sigma$ be the polyhedral fan obtained by coning over the complex below:
\begin{figure}[ht]
\begin{center}
\epsfig{file=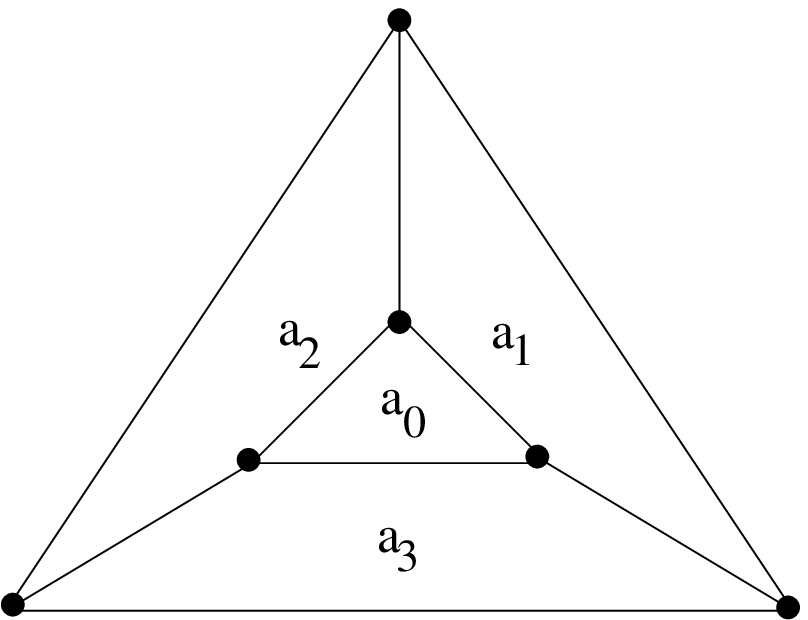,height=1.3in,width=1.5in}
\end{center}
\caption{\textsf{A two dimensional section of $\Sigma$}}
\label{fig:braid}
\end{figure}
The fan $\Sigma$ defines a three dimensional toric variety, having 4 maximal
cones, and a computation shows that the Hilbert polynomial of $C^0(\Sigma)$ is:
\[
HP(C^0(\Sigma),k) = 2k^2+2.
\]
The fan $\Sigma$ is non-generic; in the figure the three lines connecting
boundary vertices to interior vertices meet at a point. Perturbing one vertex
in the figure so the symmetry is broken yields a combinatorially equivalent fan $\Sigma'$,
with
\[
HP(C^0(\Sigma'),k) = 2k^2+1.
\]
Both of these calculations follow from Corollary~\ref{3d}.
The six interior facets of $\Sigma$ correspond to six lines in $\p^2$ defining the 
braid arrangement A$_3$. Theorem~\ref{Aroot} provides a second interpretation of the Hilbert polynomial.
\end{exm}
In Example~2.11, we use Theorem~1 to analyze the four dimensional analog 
of Example~\ref{exm:first}, showing that the module $C^0(\Sigma)$ is free with
generators in degrees $\{0,1,2,3,4\}$. Motivated by this, we define 
two families of fans associated to the reflection arrangement $\An$. One family 
corresponds to projective space, so the fans are all smooth and complete. 
The second family of fans generalizes the example above, and the fans are neither 
complete nor simplicial. We show:
\vskip .04in
\noindent{\bf Theorem 3} For both families of fans described above,
$C^0(\Sigma)$ is isomorphic to the module of logarithmic vector fields 
on the reflection arrangement A$_n$.

\pagebreak

\section{Cellular homology}\label{sec:two}
Let $\Sigma$ be a $d$-dimensional polyhedral fan embedded in $\mathbb{Q}^d$.
We make the simplifying assumption that $\Sigma$ is {\em hereditary}; this 
means for every nonempty face $\sigma$ of $\Sigma$, the dual graph of the 
star of $\sigma$ is connected. Henceforth we use {\em fan} in place 
of {\em hereditary rational polyhedral fan embedded in $\mathbb{Q}^d$}.
A $C^0$--spline on $\Sigma$ is a piecewise 
polynomial function such that two polynomials supported on $d$-faces 
which share a common $(d-1)$--face $\tau$ meet continuously along that face. 
Let $\Sigma^0_i$ denote the set of interior $i$ faces of $\Sigma$ $(\Sigma^0_d = \Sigma_d)$, and let
$f^0_i= |\Sigma^0_i| $; for a face $\tau$ write $\overline{\tau}$ for the linear span. 
The set of all $C^0$--splines on $\Sigma$ is a graded module over both $\Z[x_1,\ldots,x_d]$
and $\Q[x_1,\ldots,x_d]$. In \cite{br}, Billera-Rose observe that 
the module of splines of smoothness $r$ arises as the kernel of a 
matrix; for $r=0$ their result shows that there is a graded exact sequence:
\begin{equation}\label{BR}
0\longrightarrow C^0(\Sigma) \longrightarrow S^{f_d}\oplus
S^{f_{d-1}^0}(-1) \stackrel{\phi}{\longrightarrow}
S^{f_{d-1}^0}
\longrightarrow N \longrightarrow
0
\end{equation}
\[
 \hbox{where  } \phi = \;\;{\small \left[ \partial_d \Biggm| \begin{array}{*{3}c}
l_{\tau_1} & \  & \  \\
\ & \ddots & \  \\
\ & \ & l_{\tau_m}
\end{array} \right]}
\]
Write $[\partial_d \mid D]$ for $\phi$. To describe $\partial_d$, note 
that the rows of $\partial_d$ are indexed by $\tau \in \Sigma_{d-1}^0$. 
If $a_1, a_2$ denote the $d-$faces adjacent to $\tau$, 
then in the row corresponding to
$\tau$ the continuity condition means that the only nonzero entries
occur in the columns corresponding to $a_1, a_2$, and are
$\pm(+1,-1)$. When $\Sigma$ is simplicial, 
$\partial_d$ is the top boundary map in the (relative) chain complex. 
By Equation~\ref{BR}, $C^0(\Sigma)$ is a second syzygy, so the associated
sheaf is a reflexive $\mathcal{O}_{\p^{d-1}}$--module,
and by Hirzebruch-Riemann-Roch the coefficients of $HP(C^0(\Sigma),k)$ yield the 
Chern classes of $C^0(\Sigma)$. In \cite{ms}, the associated primes of 
$N$ are analyzed using the graph appearing in Definition~\ref{defn:dualG2},
in the special case $i=1$. Since 
\[
N \simeq \Big(\oplus_{\tau \in \Sigma_{d-1}^0}S/I_{\tau} \Big) / \partial_{d},
\]
this suggests constructing a chain complex. For each $k$-dimensional face 
$\xi \prec \Sigma$, let $I_{\xi}$ be the ideal of 
$\overline{\xi}$. In \cite{s}, a chain complex $\mathcal{C}(\Delta)$ is defined 
for a simplicial fan $\Delta$, such that the top homology module of 
$\mathcal{C}(\Delta)$ computes splines of a fixed order of smoothness on $\Delta$.
Using cellular homology, we next carry out a similar construction for polyhedral $\Sigma$.
\begin{defn}\label{compC} For a polyhedral fan $\Sigma$, let $\mathcal{C}$ be the 
complex of $S=\Q[x_1,\ldots,x_d]$ modules, with cellular differential $\partial_i$.
\[
0 \longrightarrow \bigoplus\limits_{\sigma \in \Sigma_d} S 
\stackrel{\partial_d}{\longrightarrow} \bigoplus\limits_{\tau \in \Sigma_{d-1}^0} S/I_{\tau}  
\stackrel{\partial_{d-1}}{\longrightarrow} \bigoplus\limits_{\psi \in \Sigma_{d-2}^0} S/I_{\psi}  
\stackrel{\partial_{d-2}}{\longrightarrow} \ldots
\stackrel{\partial_{2}}{\longrightarrow} \bigoplus\limits_{v \in \Sigma_{1}^0} S/I_{v}\longrightarrow 0.
\]
\end{defn}
By construction, $H_d(\mathcal{C})=C^0(\Sigma)$, and taking Euler characteristics yields
\begin{prop}\label{euler}
\[
HP(C^0(\Sigma),k) = \sum_{i=1}^d (-1)^{d-i} f_{i}^0 \cdot {k+i-1 \choose i-1}
                   +\sum_{i=1}^{d-1}(-1)^{d-1-i}   HP(H_i(\mathcal{C}),k).
\]
\end{prop}

\pagebreak

\begin{lem}\label{CdimPrimes}
For $i \ge 1$, $H_{d-i}(\mathcal{C})$ is supported in codimension at least $i+1$.
\end{lem}
\begin{proof}
Since
\[
H_{d-i}(\mathcal{C}) = H \Big( \bigoplus\limits_{\alpha \in \Sigma_{d-i+1}}S/I_\alpha 
\stackrel{\partial_{d-i+1}}{\longrightarrow} \bigoplus\limits_{\beta \in \Sigma_{d-i}^0}\!\!\!\! S/I_{\beta}  
\stackrel{\partial_{d-i}}{\longrightarrow} \bigoplus\limits_{\gamma \in \Sigma_{d-i-1}^0}\!\!\!\! S/I_{\gamma} \Big),
\]
localizing at a prime $P$ of codimension $i$ forces the vanishing of the middle
term, unless $P = I_{\beta}$. But in this case, there exists $\alpha \in \Sigma_{d-i+1}^0$ with 
$\beta \subseteq \alpha$, and so the localized map $\partial_{d-i+1}$ is onto: though there can be
several distinct $\beta_i \in \Sigma_{d-i}^0$ with $\overline{\beta_i} = V(P)$, convexity implies
that a $d-i+1$ face $\alpha$ containing one such $\beta_i$ cannot contain any others. Thus 
$H_{d-i}(\mathcal{C})_P=0$ if $codim(P) \le i$. 
\end{proof}
\begin{lem}\label{LinPrimes}
All associated primes of $H_{d-i}(\mathcal{C})$ of codimension $i+1$ are linear.
\end{lem}
\begin{proof}
Let $P$ be an associated prime of codimension $i+1$, minimally
generated by $\langle f_1,\ldots, f_m \rangle$. If less than $i+1$ of the $f_j$ are linear,
then Lemma~\ref{CdimPrimes} shows 
\[
\Bigg( \bigoplus\limits_{\beta \in \Sigma_{d-i}^0}\!\!\!\! S/I_{\beta} \Bigg)_P =0,
\]
unless $P$ contains
exactly $i$ linear forms such that $\langle f_1,\ldots, f_i \rangle = I_{\beta}$, and in 
this case, the localization of $\partial_{d-i+1}$ is surjective. So if $H_{d-i}(\mathcal{C})_P \ne 0$,
then $P$ must have at least $i+1$ minimal generators which are linear. Since $\codim(P)=i+1$,
$P= \langle l_1,\ldots, l_{i+1} \rangle$.
\end{proof}
\begin{defn}\label{defn:dualG2}
Let $\Sigma$ be a $d-$dimensional fan, and $\xi$ a $\codim(i+1)$ linear subspace.
Define a graph $G_\xi(\Sigma)$, with a vertex 
$v_\beta$ for every $\beta \in \Sigma_{d-i}^0$ with $\xi \in \overline{\beta}$ and an 
edge $\overline{v_{\beta}v_{\beta'}}$ if there exists $\alpha \in \Sigma_{d-i+1}^0$ with 
$\partial_{d-i+1}(\alpha)= \beta \pm \beta' + \cdots$ and $\xi = \overline{\beta} \cap \overline{\beta'}.$
\end{defn}
\begin{thm}\label{minassPrimes}
For $\xi$ a linear subspace of codimension $i+1$, $I_{\xi}$ is an associated prime of 
$H_{d-i}(\mathcal{C})$ iff $G_\xi(\Sigma) = \amalg \Delta_j$ has a component 
$\Delta_j$ such that the following two conditions hold:
\begin{enumerate}
\item $\Delta_j$ does not correspond to a loop around $\gamma \in \Sigma_{d-i-1}^0$.
\item $\Delta_j$ has no vertex of valence one.
\end{enumerate}
\end{thm}
\begin{proof}
Localize $\mathcal{C}$ at ${I_{\xi}}$, and write the localized complex as $\mathcal{C}_\xi$. 
Case 1: Suppose some (possibly several) $\gamma \in \Sigma_{d-i-1}^0$
have $\gamma \in \xi$. In the localized complex, we have
\[
\ldots \stackrel{\partial_{d-i+1}}{\longrightarrow} 
\bigoplus\limits_{\stackrel{\xi \prec \overline{\beta}}{\beta \in \Sigma_{d-i}^0}}\!\!\!\! (S/I_{\beta})_{I_\xi}  
\stackrel{\partial_{d-i}}{\longrightarrow} \bigoplus\limits_{\stackrel{\xi =\overline{\gamma}}{\gamma \in \Sigma_{d-i-1}^0}}\!\!\!\!\! (S/I_{\gamma})_{I_\xi} \longrightarrow 0.
\]
A codimension $i$ face $\beta$ can contain at most one facet $\gamma$ of the form above, and 
choosing any such facet shows that (in the localization) $\partial_{d-i}$ is surjective. This 
yields
\[
0 \longrightarrow H_{d-i}(\mathcal{C}_\xi) \longrightarrow \Bigg(\!\!\bigoplus\limits_{\stackrel{\xi \prec \overline{\beta}}{\beta \in \Sigma_{d-i}^0}}\!\!\!\! (S/I_{\beta})_{I_\xi} \Bigg)/im(\partial_{d-i+1}) 
\stackrel{\partial_{d-i}}{\longrightarrow} \bigoplus\limits_{\stackrel{\xi =\overline{\gamma}}{\gamma \in \Sigma_{d-i-1}^0}}\!\!\!\!\! (S/I_{\gamma})_{I_\xi} \longrightarrow 0.
\]
The map $\partial_{d-i}$ splits in $\mathcal{C}_\xi$. To see this, 
note that in $\mathcal{C}_\xi$, two types of codimension $i$ face $\beta$ 
can appear: those which actually have a $\codim(i+1)$ face $\gamma$ lying in $\xi$, and 
those which do not. In the first situation, fix $\gamma \in \Sigma^0_{d-i-1}$ such 
that $\gamma \in \xi$. By convexity, there exists a codimension $i-1$ face
$\alpha$ with $\partial_{d-i+1}(\alpha)=\beta_1 \pm \beta_2 \pm \cdots$, and 
$\gamma = \beta_1 \cap \beta_2$. Looping around $\gamma$ yields an isomorphism 
$\oplus_{\gamma \in \beta \in \Sigma_{d-i}^0}\!\!\!\! (S/I_{\beta})_{I_\xi} /im(\partial_{d-i+1}) 
\simeq (S/I_\gamma)_{I_\xi} $, and $H_{d-i}(\mathcal{C}_\xi)$ vanishes. 
In the latter situation, the localized map $\partial_{d-i}$ is zero, so
\begin{equation}\label{Hpres}
H_{d-i}(\mathcal{C}_\xi) \simeq \bigoplus\limits_{\stackrel{\xi \prec \overline{\beta}, \beta \in \Sigma_{d-i}^0}{\xi \ne \overline{\gamma} \mid \gamma \in \partial(\beta)}} \!\!\!\! (S/I_{\beta})_{I_\xi} /\im(\partial_{d-i+1}).
\end{equation}
\vskip -.05in
Case 2: no $\gamma \in \Sigma_{d-i-1}^0$ satisfies $\gamma \in \xi$.
Since no $\gamma \in \xi$, 
\[
\Big( \bigoplus\limits_{\gamma \in \Sigma_{d-i-1}^0}\!\!\!\!\! S/I_{\gamma}\Big)_{I_\xi}=0.
\]
Hence, $H_{d-i}(\mathcal{C}_\xi)$ is as in Equation~\ref{Hpres}. By convexity, a codimension $i-1$ face $\alpha$ can have
at most two facets $\beta_i$ with $\xi \prec \overline{\beta_i}$. Thus, the localized map $\partial_{d-i+1}$
will have columns with either one or two nonzero entries. In the former case, the corresponding 
generator of $S/I_{\beta}$  is killed by $\partial_{d-i+1}$. In the latter case, quotienting by a 
column with two nonzero entries gives a cokernel of the form $S/I_{\beta_1}+ I_{\beta_2}$, where
$\beta_1$ and $\beta_2$ are the two codimension $i$ faces with $\xi \prec \overline{\beta_i}$.
Thus, $I_{\beta_1}+ I_{\beta_2} = I_\xi$. If $G_\xi(\Sigma) = \amalg \Delta_j$, then
a component $\Delta_j$ has a vertex with valence one exactly when the localized 
map $\partial_{d-i+1}$ has column with single nonzero entry. Let $a_{\xi}$ denote the number of $\Delta_i$ with no vertex of valence one. 
We have shown $H_{d-i}(\mathcal{C}_\xi) \simeq (S/I_{\xi})_{I_\xi}^{a_{\xi}},$
and the result follows. 
\end{proof}
\begin{cor}\label{Hterm}
Let $L_i$ denote the Grassmannian of $\codim(i)$ linear subspaces. Associate to a point of $L_i$ 
the corresponding subspace, let $a_{\xi}$ be as above, and 
\[
\alpha_i = \sum\limits_{\xi \in L_{i+1}}  a_{\xi}
\]
Then
\[
HP(H_{d-i}(\mathcal{C}),k) = \alpha_i \cdot {k+d-i-2 \choose d-i-2} + O(k^{d-i-3}).
\] 
\end{cor}
\begin{cor}\label{Topterm}
$\alpha_1 = \big(\sum\limits_{\xi \in L_2} \rank H_1(G_\xi(\Sigma), \Z)\big) -f_{d-2}^0$.
\end{cor}
\begin{proof}For $\xi \in L_2$, the components $\Delta_i$ are homeomorphic to circles or segments.
\end{proof}
\begin{cor}\label{3d}
For a fan $\Sigma \in \Z^3$, 
\[
HP(C^0(\Sigma),k) = f_3 \cdot {k+2 \choose 2}-f_2^0(k+1) + f_1^0 + \alpha_1.
\]
\end{cor}

\begin{exm}\label{exm:third}Consider a four-dimensional version of Example~\ref{exm:first}: 
place a small regular tetrahedron $T$ symmetrically inside a large regular tetrahedron $T'$, and connect corresponding vertices. The maximal faces are the 
convex hulls of the corresponding facets of $T$ and $T'$.  
$C^0(\Sigma)$ is a free $S$--module, with generators in degrees $\{0,1,2,3,4\}$, 
so the Hilbert polynomial is
\[
\sum\limits_{i=0}^4 {k+3-i \choose 3} = 5\cdot {k+3 \choose 3}-10 {k+2 \choose 2}+10{k+1 \choose 1} -5.
\]
If $v$ is the central point of symmetry, $G_v(\Sigma)$ is the 1-skeleton of a tetrahedron, 
so $HP(H_2(\mathcal{C}),k) = HP(S/I_v,k)=1$. Since $f_0 = 4$, this yields the expected constant term in the Hilbert polynomial.
\end{exm}

\section{Cartan-Eilenberg spectral sequence}\label{sec:three}
Write the complex $\mathcal{C}$ in the first quadrant, with $\mathcal{C}_i$ appearing 
in position $(i,0)$. Taking a Cartan-Eilenberg resolution for $\mathcal{C}$ and applying the
functor $Hom_S(\bullet, S)$ yields a first quadrant double complex. For the vertical 
filtration, ${}_v\!\E^1_{ij}$ is $\Ext_S^j(\mathcal{C}_i,S)$, so since the 
$S/I_{\tau}$ are complete intersections
$$
{}_v\!\E^1_{ij} = 
\begin{cases}
\bigoplus\limits_{\tau \in \Sigma_{i}^0}\!\!\!\! S(i)/I_{\tau}& \text{if $j=d-i$;}\\
0 & \text{otherwise.}
\end{cases}
$$
Thus ${}_v\!\E^1_{ij} = {}_v\!\E^{\infty}_{ij}$. Write $G_{ij}$ for the $j^{th}$ 
module in a free resolution of $H_i(\mathcal{C})$; from the Cartan-Eilenberg construction it follows that for the horizontal filtration
\[
{}_h\!\E^1_{ij} = \Hom_S(G_{ij},S).
\]
Thus, 
\[
{}_h\!\E^2_{ij} = \Ext_S^j(H_i(\mathcal{C}),S).
\]
\begin{exm}\label{4d}
Let $d=4$. Then the ${}_v\!\E^1_{i,j}$ terms are
\begin{center}
\begin{tabular}{|c|c|c|c|c|c|}
        \hline
  &i=0   & 1 & 2 & 3 & 4  \\
        \hline
j=3 &  & $\bigoplus\limits_{v \in \Sigma_{1}^0} R/I_{v}$ & & &  \\
        \hline
j=2 &  &  &$\bigoplus\limits_{u \in \Sigma_{2}^0} R/I_{\psi}$  & &  \\
        \hline
j=1 &  &  & &$\bigoplus\limits_{\tau \in \Sigma_{3}^0} R/I_{\tau}$   &  \\
        \hline
j=0 &  &  &  & & $\bigoplus\limits_{\sigma \in \Sigma_4} R $  \\
        \hline
\end{tabular}
\end{center}
By Lemma~\ref{CdimPrimes}, for $i \ge 1$, $H_{d-i}(\mathcal{C})$ is supported 
in codimension at least $i+1$, so for the horizontal filtration, 
the ${}_h\!\E^2_{ij}$ terms are
\begin{center}
\begin{tabular}{|c|c|c|c|c|c|}
        \hline
  &i=0   & 1 & 2 & 3 & 4  \\
        \hline
j=4 & &$\Ext^4(H_1(\mathcal{C}),S)$  &$\Ext^4(H_2(\mathcal{C}),S) $ &$\Ext^4(H_3(\mathcal{C}),S)$ &  $\Ext^4(C^0(\Sigma),S)$ \\
        \hline
j=3 &  &                           &$\Ext^3(H_2(\mathcal{C}),S)$ &$\Ext^3(H_3(\mathcal{C}),S)$ & $\Ext^3(C^0(\Sigma),S)$  \\
        \hline
j=2 &  &  &  &$\Ext^2(H_3(\mathcal{C}),S)$  &$\Ext^2(C^0(\Sigma),S)$   \\
        \hline
j=1 &  &  & &  & $\Ext^1(C^0(\Sigma),S)$ \\
        \hline
j=0 &  &  &  & & $C^0(\Sigma)^\vee$   \\
        \hline
\end{tabular}
\end{center}
\end{exm}
\begin{thm}\label{SSvanish}
If $H_i(\mathcal{C})$ is Cohen-Macaulay with $\codim(H_i(\mathcal{C}))= d-i+1$ for all $i<d$, 
then $C^0(\Sigma)$ is free.
\end{thm}
\begin{proof}
By local duality, $Ext^i(M,S)$ vanishes if $i<\codim(M)$ and $i>\pdim(M)$. The Cohen-Macaulay 
condition implies $\codim(M)=\pdim(M)$, which combined with the assumption that 
$\codim(H_i(\mathcal{C}))= d-i+1$ implies that
$$
{}_h\!\E^2_{ij} =
\begin{cases}
Ext^j(H_i(\mathcal{C}),S)& \text{if $j+i=d+1$;}\\
0 & \text{otherwise.}
\end{cases}
$$
Thus, the $d_2$ and higher differentials from the terms ${}_h\!\E^2_{dj}$ must
all vanish. Comparing to the vertical filtration shows that $H_i(Tot)$ is concentrated
in degree $d$, so that ${}_h\!\E^2_{dj} = Ext^j(C^0(\Sigma),S) =0$ for $j>0$. Note that
the condition $\codim(H_i(\mathcal{C}))= d-i+1$ can be weakened to include the case $H_i(\mathcal{C})=0$.
\end{proof}

\begin{cor}\label{euler2}
If $H_i(\mathcal{C})$ is Cohen-Macaulay with either $\codim(H_i(\mathcal{C}))= d-i+1$ or 
$H_i(\mathcal{C})= 0$ for all $i<d$, then $($letting $\alpha_0=\alpha_{-1}=0)$
\[
HP(C^0(\Sigma),k) = \sum_{i=1}^d (-1)^{d-i}(f_{i}^0+\alpha_{d-1-i})\cdot {k+i-1 \choose i-1}. 
\]
\end{cor}
When $\Sigma$ is simplicial, Lemma~\ref{CdimPrimes} can be strengthened
\cite{s} to show that for $i \ge 1$, $H_{d-i}(\mathcal{C})$ is supported in 
codimension at least $i+2$, and that
\[
C^0(\Sigma)\mbox{ is free iff } H_i(\mathcal{C}) = 0 \mbox{ for all } i< d.
\]
Related results of Franz-Puppe appear in \cite{FP}, \cite{FP2}. 
In Example~\ref{exm:first}, $C^0(\Sigma)$ and $C^0(\Sigma')$ are both free; for 
$\Sigma'$ all the lower homology modules of $\mathcal{C}$ vanish, while for
$\Sigma$, $H_2(\mathcal{C}) \ne 0$. Even for $\Sigma$ simplicial, there are
easy examples where $C^0(\Sigma)$ is nonfree.
\begin{exm}
Consider the cone over the planar complex $\Sigma$ below, where the
(cone over) the central triangle has been removed:
\begin{figure}[ht]
\begin{center}
\epsfig{file=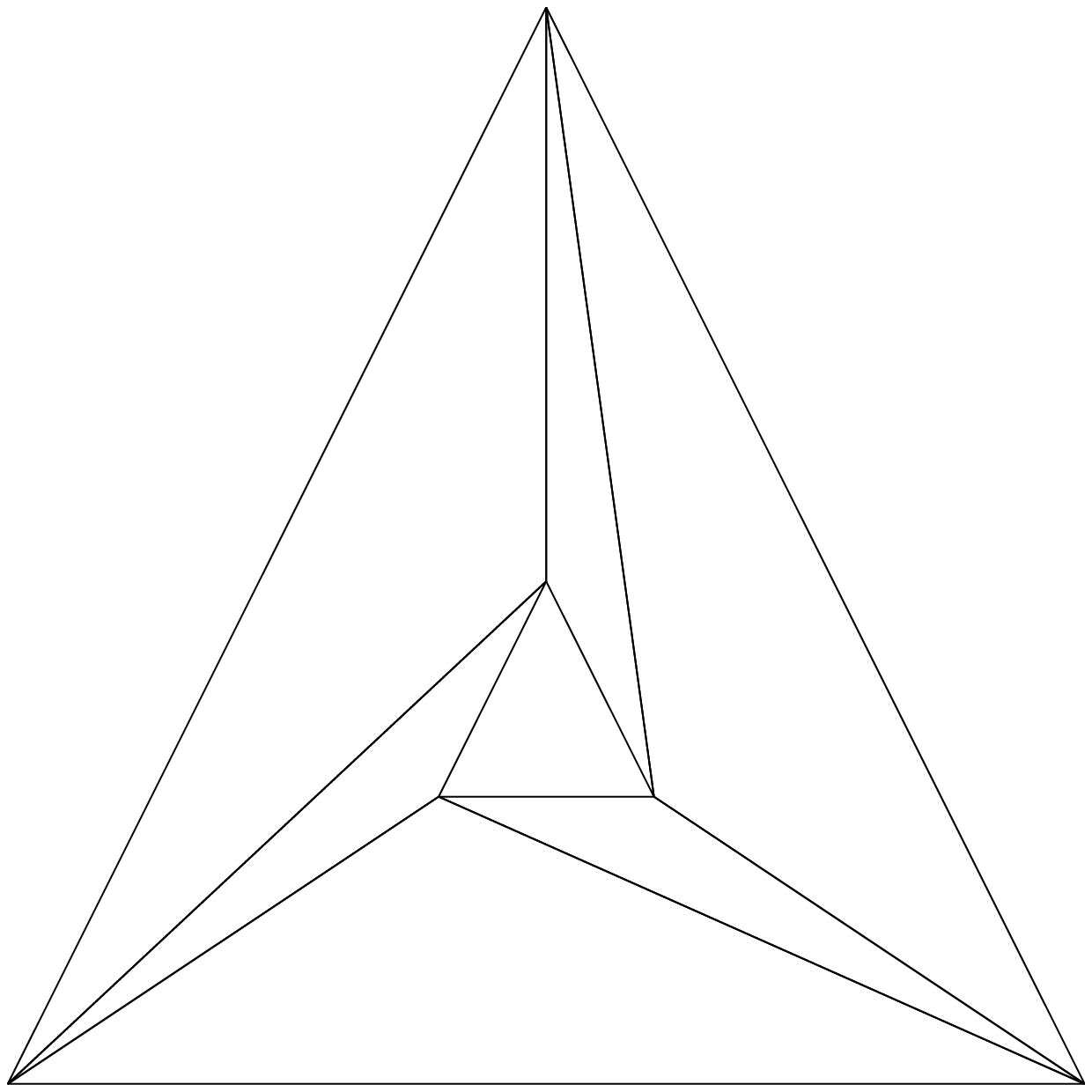,height=1.3in,width=1.5in}
\end{center}
\caption{\textsf{A simplicial nonfree $\Sigma$}}
\label{fig:nonfree}
\end{figure}

The topological nontriviality of $\Sigma$ manifests in the nonvanishing of
$H_2(\mathcal{C})$, which has high codimension: $Ext^2(H_2(\mathcal{C}))=0$, but
$Ext^3(H_2(\mathcal{C})) \ne 0.$ Comparing terms in the spectral sequence shows that this forces
\[
Ext^1(C^0(\Sigma),S)) \ne 0.
\]
$C^0(\Sigma)$ is reflexive sheaf on $\p^2$, so is locally free. The above shows it does not split.
\end{exm}
\begin{defn}\label{shell} A pure dimensional polyhedral complex is shellable if there
exists an ordering of facets $\{P_1,\ldots,P_k\}$ so that $\partial(P_1)$ is shellable,
and $P_i \cap (\cup_{j<i}P_j)$ is the start of a shelling of $\partial(P_i)$
\end{defn}
If $P$ is simplicial, an induction shows that shellability implies the 
Stanley-Reisner ring $A_\Delta$ is Cohen-Macaulay \cite{bh}. For
$\Sigma$ a pure polyhedral complex, if $\Sigma'$ is obtained from 
$\Sigma$ by removing a full dimensional cell $P$, then all faces of $P$
are part of the boundary of $\Sigma'$. This gives a short exact sequence of complexes
as in Definition~\ref{compC}, where $\mathcal{C}$ and $\mathcal{C}'$ correspond to $\Sigma$
and $\Sigma'$, and $\mathcal{C}''$ comes from $P$ and all its faces. This yields
a long exact sequence. 
\[
0 \longrightarrow \langle \prod\limits_{\tau \in P_{d-1}}l_\tau \rangle 
\longrightarrow C^0(\Sigma) \longrightarrow C^0(\Sigma') \longrightarrow H_{d-1}(\mathcal{C}'')
\longrightarrow \cdots
\]
\noindent In the simplicial case, $H_{d-1}(\mathcal{C}'')=0$, but this need not hold in the
polyhedral case. Nevertheless, computational evidence leads us to ask:
\begin{que}
If $\Sigma$ is shellable, is $C^0(\Sigma)$ a free $S$-module?
\end{que}
\section{Fans associated to type a root systems}\label{sec:four}
In this section, we analyze two very different families of fans associated to the reflection 
arrangement $\An$. Though the fans are quite different, for both families $C^0(\Sigma)$ is 
isomorphic to the module of vector fields tangent to $\An$. Given a fan, we associate
to it the hyperplane arrangement in $\p^{d-1}$ defined by $\overline{\tau}$ such that 
$\tau \in \Sigma_{d-1}^0$. For a collection of hyperplanes 
\[
\A = \bigcup\limits_{i=1}^nH_i \subseteq \C^d,
\]
the intersection lattice $L(\A)$ consists of the intersections of the elements of 
${\mathcal A}$; the rank of $x \in L_{\mathcal A}$ is simply the codimension 
of $x$. $\C^d$ is the lattice element $\hat{0}$; the rank one elements are the hyperplanes 
themselves.
\begin{defn}
The M\"{o}bius function $\mu$ : $L({\mathcal A}) \longrightarrow \Z$ is defined
 by $$\begin{array}{*{3}c}
\mu(\hat{0}) & = & 1\\
\mu(t) & = & -\!\!\sum\limits_{s < t}\mu(s) \mbox{, if } \hat{0}< t
\end{array}$$
The Poincar\'e polynomial is 
\[
\pi(\A,t) = \!\!\sum\limits_{x \in L({\mathcal A})}\mu(x) \cdot (-t)^{\text{rank}(x)}.
\]
\end{defn}
One of the main algebraic objects associated to an arrangement
is the graded $S$--module $D({\mathcal A})$ of vector fields
tangent to $\A$.
\begin{defn}
 $D({\mathcal A}) = \{ \theta \mid \theta(l_i) \in \langle l_i \rangle \mbox{ for all }V(l_i) \in \A \} \subseteq Der_{\Q}(S).$
\end{defn}
It follows immediately from the definition that $D(\A)$ can be computed as the 
kernel of a matrix similar to $\phi$ in Equation~\ref{BR}: If $l_j =\sum a^j_i x_i$, let $D(l)$ be the $|\A| \times d$ matrix with $j^{th}$ row $[a^j_1,\ldots, a^j_d]$. Then $D(\A)$ is the kernel of
\[
{\small \left[ D(l) \Biggm| \begin{array}{*{3}c}
l_{\tau_1} & \  & \  \\
\ & \ddots & \  \\
\ & \ & l_{\tau_m}
\end{array} \right],}
\]
\begin{defn}\label{free}
$\mathcal{A}$ is {\em free} if $D({\mathcal A}) \simeq
\oplus S(-d_i)$; the $d_i$ are the {\em exponents} of ${\mathcal A}$.
\end{defn}
In \cite{t}, Terao proves that if $D({\mathcal A}) \simeq \oplus S(-d_i)$, then $\pi(\A,t)= \prod(1+d_it)$, and
in \cite{t2} shows that if $G$ is a finite reflection group acting with no fixed points, then
the arrangement of reflecting hyperplanes $\A_G$ is free, with $d_i$ equal to the degrees of 
generators of $S^G$.
\begin{exm}\label{AN}
The reflecting hyperplanes of $SL(n)$ are 

\[
\bigcup_{1 \le i < j \le n+1}V(x_j-x_i) = \An \subseteq \C^{n+1}.\]
The symmetric group $S_n$ action fixes the 
subspace $(t,\ldots,t)$. Projecting along this subspace yields 
$\An \subseteq \C^n$, with $S_n$ acting without fixed points. By \cite{t2}, 
$\An \subseteq \C^n$ is free with exponents $\{1,2,\ldots, n \}$ and
$\An \subseteq \C^{n+1}$ is free with exponents $\{0,1,2,\ldots, n \}$.
Up to projective transformation, $A_3$ is exactly the configuration defined by 
the six $\tau \in \Sigma_2^0$ appearing in Example~\ref{exm:first}.
The corresponding $\overline{\tau}$ yield an arrangement of six planes 
thru the origin $\C^3$, depicted as a set of lines in $\p^2$:
\begin{figure}[ht]
\subfigure{%
\label{fig:braid4}%
\begin{minipage}[t]{0.35\textwidth}
\setlength{\unitlength}{16pt}
\begin{picture}(6,8)(-1,-2.5)
\put(-0.8,0){\line(1,0){5.6}}
\put(-0.8,-0.4){\line(2,1){4.7}}
\put(-0.4,-0.8){\line(1,2){2.8}}
\put(2,-0.8){\line(0,1){5.6}}
\put(4.8,-0.4){\line(-2,1){4.7}}
\put(4.4,-0.8){\line(-1,2){2.8}}
\put(-1.7,0.3){\makebox(0,0){$L_{1}$}}
\put(-1.5,-0.6){\makebox(0,0){$L_{4}$}}
\put(-0.5,-1.3){\makebox(0,0){$L_{2}$}}
\put(2,-1.3){\makebox(0,0){$L_{5}$}}
\put(4.5,-1.3){\makebox(0,0){$L_{3}$}}
\put(5.5,-0.6){\makebox(0,0){$L_{6}$}}
\end{picture}
\end{minipage}
}
\setlength{\unitlength}{0.8cm}
\subfigure{%
\label{fig:partitionlat}%
\begin{minipage}[t]{0.45\textwidth}
\begin{picture}(5,5.7)(0,-4)
\xygraph{!{0;<10mm,0mm>:<0mm,14mm>::}
[]*D(3){{\bf 0}}*-{\bullet}  
(
-@{..}[dlll]*D(3.4){124}*-{\bullet}  
(
-@{..}[d]*U(2.5){1}*-{\bullet}-@{..}[drrr]*U(2.5){\C^3}*-{\bullet} 
,-@{..}[dr]*U(2.5){4}*-{\bullet}-@{..}[drr]
,-@{..}[drr]*U(2.5){2}*-{\bullet}-@{..}[dr]
)
,-@{..}[dll]*D(2.1){34}*-{\bullet}  
(
-@{..}[d]
,-@{..}[drrrr]*U(2.5){3}*-{\bullet}-@{..}[dll]
)
,-@{..}[dl]*D(3.4){136}*-{\bullet} 
(
-@{..}[dll]
,-@{..}[drr]*U(2.5){6}*-{\bullet}-@{..}[dl]
,-@{..}[drrr]
)
,-@{..}[d]*D(2.1){26}*-{\bullet} 
(
-@{..}[dl]
,-@{..}[dr]
)
,-@{..}[dr]*D(3.4){456}*-{\bullet}  
(
-@{..}[dlll]
,-@{..}[d]
,-@{..}[drr]*U(2.5){5}*-{\bullet}-@{..}[dlll]
)
,-@{..}[drr]*D(2.1){15}*-{\bullet}  
(
-@{..}[dlllll]
,-@{..}[dr]
)
,-@{..}[drrr]*D(3.4){235}*-{\bullet}  
(
-@{..}[dllll]
,-@{..}[dl]
,-@{..}[d]
)
(
}
\end{picture}
\end{minipage}
}
\vskip -.12in
\caption{\textsf{The braid arrangement $A_{3}$ and its intersection lattice in $\C^3$}}
\label{fig:braid2}
\end{figure}
\end{exm}
\noindent In Example~\ref{exm:third}, $C^0(\Sigma)$ is free, with generators in degrees $\{0,1,2,3,4\}$,
and the associated arrangement is A$_4$. This motivates:
\begin{defn}
Let $v_i=(n+1){\bf e_i} -\sum_{i=1}^n {\bf e_i}$ and $\Delta_2 = \{ v_1,\ldots, v_n \}$.
Define cones $a_0 = \cone\{{\bf e_1},\ldots,{\bf e_n}\}$, and for $i \in \{1,\ldots,n\}$, 
$a_i = \cone\{ \Delta_2 \setminus v_i, a_0 \setminus {\bf e_i} \}$. Then 
$P_2(\An) = \{a_0, a_1,\ldots, a_n \}$ is a polyhedral fan, which is a subdivision of $\cone(\Delta_2)$.
\end{defn}
\noindent The fan $P_2(A_3)$ appears in Example~\ref{exm:first} and the fan $P_2(A_4)$ in Example~\ref{exm:third}. 
\begin{thm}\label{Aroot}
$C^0(P_2(\An)) \simeq D(\An)$. 
\end{thm}
\begin{proof}
The main point is that the passage from 
$n$ to $n+1$ induces the same modification in the matrices which 
compute $D(\An)$ and $C^0(P_2(\An))$. A computation shows that the
hyperplanes $\overline{\tau}$ for $\tau \in  P(\An)^0_{d-1}$ are of the
form $V(x_i)$ or $V(x_j -x_i)$ for $1 \le i<j \le n$. 
Order the $d$-faces of $P_2(\An)$ as $\{a_d, a_{d-1}, \ldots,a_0 \}$ 
(see Figure~1), and
order $\tau \in \Sigma_{d-1}^0$ so that the $x_i$ are the first $n$ elements, and the
remaining hyperplanes form a second block. For example, the lines $L_1,L_2,L_3$ in Figure~3 
correspond to the boundary of the inner triangle in Figure~1, and lines $L_4,L_5,L_6$ 
correspond to the $\tau \in \Sigma^0_2$ defining $V(x_i-x_j)$. 
Thus, the matrix which computes $C^0(P_2(\Aroot_3))$ has the form:
\[
 {\small \left[ \begin{array}{*{4}c}
-1 & 0&0 &1  \\
0 & -1 &0 &1\\
0&0 & -1 &1\\
-1&0& 1& 0\\
0& -1&1& 0\\
-1 &1 & 0 &0
\end{array}
\begin{array}{*{6}c}
x_1 &  0  & 0 & 0  &0  &0 \\
 0   & x_2  &0 & 0  &0  & 0\\
 0   &0   &x_3  &0   & 0 &0 \\
  0  & 0  & 0 &x_3-x_1   &0  &0 \\
 0   & 0  &0  & 0  & x_3-x_2 &0 \\
0    & 0  &0  & 0  &0  & x_2-x_1\\
\end{array} \right]}
\]
If we consider $\Aroot_3 \subseteq \C^4$, and order so that the equations $x_4-x_i$ appear
first, then the matrix which computes $D(\Aroot_3)$ has the form:
\[
 {\small \left[ \begin{array}{*{4}c}
-1 & 0&0 &1  \\
0 & -1 &0 &1\\
0&0 & -1 &1\\
-1&0& 1& 0\\
0& -1&1& 0\\
-1 &1 & 0 &0
\end{array}
\begin{array}{*{6}c}
x_4-x_1 &  0  & 0 & 0  &0  &0 \\
 0   & x_4-x_2  &0 & 0  &0  & 0\\
 0   &0   &x_4-x_3  &0   & 0 &0 \\
  0  & 0  & 0 &x_3-x_1   &0  &0 \\
 0   & 0  &0  & 0  & x_3-x_2 &0 \\
0    & 0  &0  & 0  &0  & x_2-x_1\\
\end{array} \right]}
\]
Since $\Aroot_3 \subseteq \C^4$ contains the subspace $(t,t,\ldots,t)$, 
$D(\Aroot_3)$ has a free summand of degree zero. To apply \cite{t2},
$G$ must act with no fixed points. Specializing to $x_4=0$ yields 
$\Aroot_3 \subseteq \C^3$ with this property. The general case follows by induction.
\end{proof}

\begin{cor}\label{Ak}
As an $S$--module, 
$A^*_T(X_{P_2(\An)})_{\Q} \simeq \bigoplus\limits_{i=0}^nS(-i).$
\end{cor}
\begin{prop}\label{Aroot2}
Let $P_1(\An)$ denote the fan such that $X_{P_1(\An)} = \p^n$. Then
\[
C^0(P_1(\An)) \simeq D(\An). 
\]
\end{prop}
\begin{proof} $P_1(\An)$ has defining hyperplanes A$_n$, and $n+1$ top 
dimensional cones; order the cones in the same way as for $P_2(\An)$. Intuitively, 
the correspondence comes from a ``mirror'' Schlegel diagram for an $n$-simplex, where the
facets are reflected outwards from the hyperplane of projection.
\end{proof}
\begin{exm}
The fan $P_1(\An)$ is simplicial, with Stanley-Reisner ring 
$\Z[x_0,\ldots, x_n]/\langle x_0\cdots x_n \rangle$. This is easily seen to have Hilbert series 
\[
\frac{1-t^{n+1}}{(1-t)^{n+1}} = \sum_{i=0}^n t^i/(1-t)^n,
\]
which is indeed equal to the Hilbert series of $D(\An) \simeq \oplus_{i=0}^nS(-i)$. 
By Corollary~\ref{Ak}, this is also equal to $HS(A^*_T(X_{P_2(\An)})_{\Q},t)$.
\end{exm}

\begin{que}\label{otherRoots}
Is there an analog of Theorem~\ref{Aroot} for other classical groups?
\end{que}

\noindent {\bf Acknowledgments}:  Macaulay2 computations were essential to this work.
I thank Michael Dipasquale, Sam Payne and Vic Reiner for useful conversations.
\renewcommand{\baselinestretch}{1.0}
\small\normalsize 
\bibliographystyle{amsalpha}

\end{document}